\documentclass[12pt]{article}
\usepackage{amsfonts}

\usepackage{amsfonts}
\usepackage{}
\usepackage{bbm}
\usepackage [dvips]{graphics}
\usepackage[centertags]{amsmath}
\usepackage{amsfonts}
\usepackage{amssymb}
\usepackage{amsthm}
\usepackage{newlfont}
\usepackage{stmaryrd}
\usepackage[]{titletoc}  
\usepackage{mathrsfs}
\usepackage{stmaryrd}
\usepackage [dvips]{graphics}
\usepackage[centertags]{amsmath}
\usepackage{amsfonts}
\usepackage{amssymb}
\usepackage{amsthm}
\usepackage{newlfont}
\usepackage{hyperref}

\newlength{\defbaselineskip}
\newcommand{\setlinespacing}[1]%
           {\setlength{\baselineskip}{#1 \defbaselineskip}}

\theoremstyle{plain}
\newtheorem{thm}{Theorem}[section]

\newtheorem{prop}[thm]{Proposition}

\theoremstyle{definition}

\newtheorem{rem}{Remark}[section]

\makeatletter\@addtoreset{equation}{section} \makeatother

\begin{document}
\title{The Maximum Principle for Global Solutions of Stochastic Stackelberg Differential Games\thanks{The first author is supported by WCU (World Class University) program through the National Research Foundation of Korea funded by the Ministry of Education, Science and Technology (R31 - 20007) and by the Research Grants Council of HKSAR (PolyU 5001/11P). The second author is supported by NNSF of China (Grant No.11101140).}}
\author{Alain Bensoussan\footnote{Naveen Jindal School of Management, The University of Texas at Dallas, Richardson, TX, USA.} \footnote{Graduate School of Business, The Hong Kong Polytechnic University, Hong Kong, and Graduate Department of Financial Engineering, Ajou University, Suwon, South Korea.}\ , Shaokuan Chen$^\dag$ and Suresh P. Sethi$^\dag$}
\maketitle  \noindent \textbf{Abstract:}
This paper obtains the maximum principle for both stochastic (global) open-loop and stochastic (global) closed-loop Stackelberg differential games. For the closed-loop case, we use the theory of controlled forward-backward stochastic differential equations to derive the maximum principle for the leader's optimal strategy. In the special case of the open-loop linear quadratic Stackelberg game, we consider the follower's Hamiltonian system as the leader's state equation, derive the related stochastic Riccati equation, and show the existence and uniqueness of the solution to the Riccati equation under appropriate assumptions. However, for the closed-loop linear quadratic Stackelberg game, we can write the related Riccati equation consisting of forward-backward stochastic differential equations, while leaving the existence of its solution as an open problem.

\section{Introduction}\label{sec1}
In 1934, H. von Stackelberg introduced a concept of a hierarchical
solution for markets where some firms have power of domination over
others \cite{Stackelberg34}. This solution concept is now known as
the Stackelberg equilibrium or the Stackelberg solution which,  in
the context of  two-person nonzero-sum static games, involves
players with asymmetric roles, one leading (called the leader) and
the other following (called the follower). A Stackelberg game
proceeds with the leader announcing his policy prior to the start of
the game. With the knowledge of the leader's strategy, the follower
chooses a policy so as to optimize his own performance index. The
leader, anticipating the follower's optimal response, picks the
policy which optimizes his performance index on the rational
reaction curve of the follower, which together with the
corresponding policy of the follower is known as the Stackelberg
solution.

In dynamic Stackelberg games, it becomes important to know the
player's information sets at any given time. In this paper, we will
consider two different information structures: i) open-loop for both
players and ii) closed-loop perfect state (CLPS) for both players.
Moreover, we will only treat global solution where the leader
announces his entire strategy at the start of the game and the
follower reacts to the entire strategy. The solutions of games with
the first information structure will be termed (global) open-loop
Stackelberg solutions, whereas the solutions of the games with the
second information structure will be termed (global) closed-loop
Stackelberg solutions. It is known that both these solutions suffer
from time inconsistency, which results from the functional
dependence of the follower's optimal response strategy on the
leader's entire strategy on the duration of the game.

In addition to these concepts, there is another concept of feedback Stackelberg
solution, where the Stackelberg property is retained at every stage
(in the discrete-time setting) with the leader having only stagewise
advantage over the follower. Since the continuous-time problem can
be viewed as the number of stages becomes unbounded in any finite
interval, stagewise advantage of the leader over the follower turns
into instantaneous advantage. A good aspect of this solution is that
it is time consistent. Readers interested in the theory and applications of this solution can refer to
\cite{BasarHaurie84}, \cite{Bensoussanetal12}, \cite{Dockner et al00}, \cite{HePrasadSethi09},
\cite{Heetal07} and \cite{KoganTapiero07}.

In an open-loop or closed-loop Stackelberg differential game, the
follower aims at minimizing his cost functional in accordance with
the leader's strategy on the whole duration of the game.
Anticipating the follower's optimal response depending on his entire
strategy, the leader chooses an optimal one in advance to
minimize his own cost functional, based on the Hamiltonian system
satisfied by the follower's optimal response.
The difference between the two kinds of games is whether the
information sets of the players involve the history of the state.
The introduction of the history of the state in the closed-loop
Stackelberg game, even in the deterministic case, makes it difficult
to tackle, as the follower may not obtain his optimal response if the
leader's announced strategy incorporates the memory of the state.
Two approaches to circumvent this difficulty are introduced: the team approach and
the maximum principle. For the former, one can refer to \cite{Basar79a}, \cite{BasarSelbuz79b} in the discrete-time setting and \cite{Papavassi79}, \cite{PapavassiCruz79b},
\cite{PapavassiCruz80} and \cite{BarsarOlsder80} in the continuous-time setting. For the latter, one can refer to \cite{PapavassiCruz79} for nonclassical control 
problems arising from Stackelberg games. The idea of team approach is as follows: the leader first minimizes his cost functional over the controls of
both the leader and the follower, yielding a lower bound on his
cost functional and the team strategies for both players. Then the
leader makes an effort to find a closed-loop strategy such that the
follower's optimal response and the state trajectory will coincide
with his team strategy and the team optimal trajectory, which leads
to the lower bound on the leader's cost functional. The maximum principle approach
restricts the leader's strategy to depend only on the initial state
and the current state (memoryless perfect state information
structure) and a nonclassical control problem faced by the leader is
solved.
It is worth noting that in this case, the follower's adjoint equation
involves the derivative of the leader's strategy with respect to the
state. Therefore, after incorporating the follower's adjoint
variable as an augmented state, the leader encounters a nonclassical
control problem with the feature that both the control and its
derivative with respect to the state appear in the controlled
forward-backward ordinary differential equation system. The authors provide two approaches to tackle this
problem and give the necessary conditions satisfied by the leader's
optimal strategy. One is to directly apply the variational technique
to the state system with mixed-boundary conditions (the adjoint
equation of the follower with a terminal condition). The other is to
establish an equivalent relationship between such a nonclassical
control problem and a classical control problem, which yields that
the optimal strategy could be found in the space of affine
functions. The phenomenon of time inconsistency is also analyzed by
the authors. We will elaborate on the technical details and generalize
their result to the stochastic setting in section \ref{sec3}.

For the stochastic formulation of Stackelberg games involving white noise terms, Yong \cite{Yong02} studies the open-loop linear quadratic case, with control variables appearing in diffusion term of the state. To give a state feedback representation of the open-loop Stackelberg solution (in a non-anticipating way), the related Riccati equation is derived and sufficient conditions for the existence of its solution with deterministic coefficients are discussed. More recently, {\O}ksendal et al \cite{Oksendaletal11} have considered a general stochastic open-loop Stackelberg differential game, proved a sufficient maximum principle, and applied the theory to continuous-time newsvendor problems.

In this paper, we study stochastic global Stackelberg differential games with open-loop and closed-loop information structures. As we shall see, the problems confronted by the leader in both cases, from the current point of view, are control problems with the state equations being forward-backward stochastic differential equations (FBSDEs). The theories for nonlinear backward stochastic differential equations (BSDEs) and FBSDEs have been extensively studied over the last two decades following the initial work by Pardoux and Peng \cite{PardouxPeng90}. One can refer to, among others, \cite{Maetal94}, \cite{MaYong99}, \cite{PardouxTang99}, \cite{PengWu99}, \cite{Yong10a}, and the references therein, for the development of the theory of FBSDEs and their applications. With the help of the results in optimization problems for controlled FBSDEs (see, e.g., \cite{ShiWu06} and \cite{Yong10b}), we obtain the maximum principle for the leader's optimal strategies in stochastic global Stackelberg games, and discuss linear quadratic problems as well as the corresponding Riccati equations.

This paper is organized as follows. In section 2 we formulate a stochastic Stackelberg game and give three types of concepts of equilibria. In section 3 we present the maximum principle for a stochastic open-loop Stackelberg game. In section 4 we focus on a stochastic closed-loop Stackelberg game and derive a maximum principle for the leader's optimal strategy. As examples, linear quadratic stochastic open-loop and closed-loop Stackelberg games are studied in section 5. For the open-loop linear quadratic case, we show the existence and uniqueness of the solution to the associated stochastic Riccati equation under some assumptions. For the closed-loop case, we simply derive a new Riccati equation consisting of FBSDEs, without investigating the issue of the existence of its solution.
\section{Problem formulation and definition of equilibria}\label{sec2}
Let $(\Omega,\mathcal {F},P)$ be a complete probability space on
which is defined a $d$-dimensional standard Brownian motion
$\{W(t),0\leq t\leq T\}$. $\{\mathcal {F}_t\}_{0\leq t\leq T}$ is
the natural filtration generated by $W$ and augmented by all the
$P$-null sets in $\mathcal {F}$ and $\mathcal {P}$ is the
predictable sub-$\sigma$-field of $\mathcal {B}([0,T])\times\mathcal
{F}$.

We consider a stochastic differential system
\begin{equation}\label{o1}
\left\{
\begin{split}
  dx(t)&=f(t,x(t),u(t),v(t))dt+\sigma(t,x(t))dW(t),\\
  x(0)&=x_0,
\end{split}\right.
\end{equation}
where
$$f:\Omega\times[0,T]\times\mathbb{R}^n\times\mathbb{R}^{m_1}\times\mathbb{R}^{m_2}\rightarrow\mathbb{R}^n,$$
$$\sigma:\Omega\times[0,T]\times\mathbb{R}^n\rightarrow\mathbb{R}^{n\times d},$$
are $\mathcal {P}\times\mathcal {B}(\mathbb{R}^{n+m_1+m_2})/\mathcal
{B}(\mathbb{R}^{n})$ and $\mathcal {P}\times\mathcal
{B}(\mathbb{R}^{n})/\mathcal {B}(\mathbb{R}^{n\times d})$
measurable, respectively, and $(u(\cdot),v(\cdot))$ are the decision
variables of the leader and the follower, respectively. The cost
functionals for the leader and the follower to minimize are
described as follows
\begin{equation*}\label{o2}
\begin{split}
  J_1(u,v)&=E[\int_0^Tg_1(t,x(t),u(t),v(t))dt+G_1(X(T))],\\
  J_2(u,v)&=E[\int_0^Tg_2(t,x(t),u(t),v(t))dt+G_2(x(T))],
\end{split}
\end{equation*}
with $$g_i:\Omega\times[0,T]\times\mathbb{R}^{n}\times U\times
V\rightarrow\mathbb{R},$$
$$G_i:\Omega\times\mathbb{R}^n\rightarrow\mathbb{R},$$ $i=1,2$,
being $\mathcal {P}\times\mathcal {B}(\mathbb{R}^{n})\times\mathcal
{B}(U)\times\mathcal {B}(V)/\mathcal {B}(\mathbb{R})$ and $\mathcal
{F}_T\times\mathcal {B}(\mathbb{R}^n)/\mathcal {B}(\mathbb{R})$
measurable, respectively.

According to the player's information sets at any given time, there
are three types of Stackelberg games: (global) open-loop, (global)
closed-loop, and feedback Stackelberg games.

\textbf{Open-loop games:} In an open-loop Stackelberg game, the leader's
information set at time $t$ is $\{x_0,\mathcal {F}_t\}$.
Therefore, the strategy $u$ announced by the leader is an $\mathcal
{F}_t$-adapted process. The follower aims at minimizing his cost
functional $J_2(u,v)$ in accordance with the leader's strategy $u$
on the whole duration of the game. His optimal response $\Phi(u)$
will be an adapted process such that
$$J_2(u,\Phi(u))\leq J_2(u,v),\ \ \forall\ u,v.$$
The leader, anticipating the follower's optimal response $\Phi$,
picks the policy $u^*$ which optimizes his performance index on the
rational reaction curve of the follower, i.e.,
$$J_1(u^*,\Phi(u^*))\leq J_1(u,\Phi(u)),\ \forall\ u.$$
$(u^*,\Phi(u^*))$ is a Stackelberg solution for an open-loop
game.

\textbf{Closed-loop games:} In a closed-loop Stackelberg game, the
information set for the leader at time $t$ is $\{\mathcal
{F}_t,x_s,s\in[0,t]\}$ (closed-loop perfect state information). The
strategy that the leader adopts now can incorporate the history
information of the state. Since in general it is difficult for the follower to
obtain his optimal response if the leader's announced strategy
incorporates the whole history of the state, we only consider the
closed-loop case under the memoryless perfect state information
pattern, i.e., the information set of the leader at time $t$ is
$\{x_0,x_t,\mathcal {F}_t\}$. For leader's each strategy
$u(t,x_0,x)$, which is now a stochastic field, the follower tries to
find his optimal response $\Psi(u)$ such that
$$J_2(u,\Psi(u))\leq J_2(u,v),\ \forall\ u,v.$$
Taking into account the follower's optimal response, the leader
should choose $u^*$ such that
$$J_1(u^*,\Psi(u^*)\leq J_1(u,\Psi(u)),\ \forall\ u.$$
$(u^*,\Psi(u^*))$ is a Stackelberg solution for a closed-loop
game.

\textbf{Feedback games:} In a feedback Stackelberg game, the
information set for the leader at time $t$ is $\{x_t,\mathcal
{F}_t\}$ (feedback pattern). The significant mechanism difference between
feedback games and the former two types of games is that the
advantage of the leader over the follower in a feedback Stackelberg
game is instantaneous not global, as the differential game could be
viewed as the limit of the discrete-time game as the number of
stages becomes unbounded (see \cite{BasarHaurie84}). Therefore,
corresponding to the leader's instantaneous strategy $u(t,x)$, the
follower will make an instantaneous response of the form $v(t,x,u(t,x))$, which
depends on the current state and the leader's current action. A
feedback solution is a pair of strategies $(u^*,v^*)$ such that
\begin{align*}
  &J_1(u^*,v^*(u^*))\leq J_1(u,v^*(u)),\ \forall\ u,\\
  &J_2(u^*,v^*(u^*))\leq J_2(u^*,v(u^*)).\ \forall\ v.
\end{align*}
From the definition we can see that the feedback Stackelberg
solution has some equilibrium feature, whereas the open-loop or
closed-loop solution involves a sequential optimization at the level
of the follower and the leader.
\section{Stochastic open-loop Stackelberg differential games}\label{open}
We first introduce some notations.
For two vectors $x$ and $y$ in
$\mathbb{R}^n$, $\langle x,y\rangle$ means the inner product
$\sum_{i=1}^nx_iy_i$. For a function $f$ defined on $\mathbb{R}^n$,
$Df$ or $\partial f$ means the gradient of $f$. Here we specify that
throughout this paper all the vectors are column vectors and the
gradient of a scalar function $f$ is $\frac{\partial f}{\partial
x}=(\frac{\partial f}{\partial x_1},\cdots,\frac{\partial
f}{\partial x_n})^\top$, while the gradient of a vector function
$f=(f_1,\cdots,f_m)^\top$ is a matrix
\begin{equation*}
\frac{\partial f}{\partial x}=\left(
  \begin{array}{ccc}
  \frac{\partial f_1}{\partial x_1}&\cdots&\frac{\partial f_1}{\partial x_n}\\
  \vdots&\vdots&\vdots\\
  \frac{\partial f_m}{\partial x_1}&\cdots&\frac{\partial f_m}{\partial x_n}\\
  \end{array}\right).
\end{equation*}
We further introduce two spaces of adapted processes to be used in the definition of the solution to a FBSDE,
\begin{align*}
  \mathcal{S}^2(0,T;\mathbb{R}^n):=\{&\psi|\ \psi:\Omega\times[0,T]\rightarrow\mathbb{R}^n\ \textrm{is a continous adapted process such that}\\
   &E\sup_{0\leq t\leq T}|\psi(t)|^2<\infty\},\\
  \mathcal{M}^2(0,T;\mathbb{R}^n):=\{&\psi|\ \psi:\Omega\times[0,T]\rightarrow\mathbb{R}^n\ \textrm{is an adapted process such that}\\
  &E\int_0^T|\psi(t)|^2dt<\infty\}.
\end{align*}
And the above two spaces will be simply written as $\mathcal{S}^2$ and $\mathcal{M}^2$, respectively, if no confusion arises.

The admissible strategy spaces for the leader and the follower are denoted by
\begin{equation*}\label{f1}
    \begin{split}
    \mathcal {U}&=\{u|u: \Omega\times[0,T]\rightarrow U\ \textrm{is}\ \mathcal {F}_t\textrm{-adapted and}\ E\int_0^T|u(t)|^2dt<+\infty\},\\
    \mathcal {V}&=\{v|v: \Omega\times[0,T]\rightarrow V\ \textrm{is}\ \mathcal {F}_t\textrm{-adapted and}\ E\int_0^T|v(t)|^2dt<+\infty\},
    \end{split}
\end{equation*}
where $U$ and $V$ are subsets of $\mathbb{R}^{m_1}$ and $\mathbb{R}^{m_2}$.

For the completeness of this paper, we state the formulation of general stochastic open-loop Stackelberg games and the corresponding maximum principle.
From the definition in section \ref{sec2}, given the leader's strategy $u\in\mathcal {U}$, the follower is
faced the stochastic control problem
$$\min_{v\in\mathcal {V}} J_2(u,v)=E[\int_0^Tg_2(t,x(t),u(t),v(t))dt+G_2(x(T))]$$
subject to
\begin{equation*}
\left\{\begin{aligned}
  dx(t)&=f(t,x(t),u(t),v(t))dt+\sigma(t,x(t))dW(t),\\
   x(0)&=x_0.
\end{aligned}\right.
\end{equation*}
Suppose there exists a unique solution $v^*(u(\cdot))\in\mathcal {V}$ to the above problem for each $u\in\mathcal {U}$. If we define
$$H_2(t,x,u,v,p_2,q_2):=\langle p_2, f(t,x,u,v)\rangle+\langle q_2,\sigma(t,x)\rangle+g_2(t,x,u,v),$$
then the maximum principle (see \cite{YongZhou99}) yields that there exists a pair of adapted processes $(p_2,q_2)\in\mathcal{S}^2\times\mathcal{M}^2$ such that
\begin{equation}\label{o3}
  \left\{
  \begin{split}
    dx(t)=&f(t,x(t),u(t),v^*(t))dt+\sigma(t,x(t))dW(t),\\
    -dp_2(t)=&\big\{(\frac{\partial f}{\partial x})^\top(t,x(t),u(t),v^*(t))p_2(t)+(\frac{\partial \sigma}{\partial x})^\top(t,x(t))q_2(t)\\
    &+\frac{\partial g_2}{\partial x}(t,x(t),u(t),v^*(t))\big\}dt-q_2(t)dW(t),\\
    x(0)=&x_0,\ \ p_2(T)=\frac{\partial G_2}{\partial x}(x(T)),\\
    v^*(t)=&arg \min_{v\in V} H_2(t,x(t),u(t),v,p_2(t),q_2(t)).
  \end{split}\right.
\end{equation}
We assume that by the last equation in \eqref{o3} a function
$v=v^*(t,x,u,p_2)$ is implicitly and uniquely defined. After
substituting $v=v^*(t,x,u,p_2)$ into the follower's maximum
principle, we get the control problem faced by the leader
$$\min_{u\in\mathcal {U}}\ J_1(u)=E[\int_0^Tg_1(t,x(t),u(t),v^*(t,x(t),u(t),p_2(t)))dt+G_1(X(T))]$$
subject to
\begin{equation}\label{04}
    \left\{
  \begin{split}
    dx(t)=&f(t,x(t),u(t),v^*(t,x(t),u(t),p_2(t)))dt+\sigma(t,x(t))dW(t),\\
    -dp_2(t)=&\big\{(\frac{\partial f}{\partial x})^\top(t,x(t),u(t),v^*(t,x(t),u(t),p_2(t)))p_2(t)+(\frac{\partial \sigma}{\partial x})^\top(t,x(t))q_2(t)\\
    &+\frac{\partial g_2}{\partial x}(t,x(t),u(t),v^*(t,x(t),u(t),p_2(t)))\big\}dt-q_2(t)dW(t),\\
    x(0)=&x_0,\ \ p_2(T)=\frac{\partial G_2}{\partial x}(x(T)).
  \end{split}\right.
\end{equation}
We denote
\begin{equation}\label{05}
  \begin{split}
    &H_1(t,u,x,y,p_1,p_2,q_1,q_2)\\
=&\langle p_1, f(t,x,u,v^*(t,x,u,p_2))\rangle+\langle q_1,\sigma(t,x)\rangle+g_1(t,x,u,v^*(t,x,u,p_2))\\
&-\langle y,(\frac{\partial f}{\partial x})^\top(t,x,u,v^*(t,x,u,p_2))p_2+(\frac{\partial \sigma}{\partial x})^\top(t,x)q_2+\frac{\partial g_2}{\partial x}(t,x,u,v^*(t,x,u,p_2))\rangle.
  \end{split}
\end{equation}
Suppose $u^*$ is an optimal strategy for the leader. Then the
maximum principle for controlled forward-backward stochastic
differential equations (see, e.g., \cite{ShiWu06} or \cite{Yong10b})
yields that there exists a triple of adapted processes $(p_1,q_1,y)$
such that
\begin{equation}\label{oo6}
  u^*(t)=\arg\min H_1(t,u,x(t),y(t),p_1(t),p_2(t),q_1(t),q_2(t)),
\end{equation}
and
\begin{equation}\label{o6}
  \left\{
  \begin{split}
    dy(t)=&-\frac{\partial H_1}{\partial p_2}dt-\frac{\partial H_1}{\partial q_2}dW(t),\\
    =&-\{(\frac{\partial f}{\partial v}\frac{\partial v^*}{\partial p_2})^\top p_1-\frac{\partial f}{\partial x}y-\sum_{i=1}^{n}y_i(\frac{\partial v^*}{\partial p_2})^\top\frac{\partial}{\partial v}(\frac{\partial f}{\partial x_i})^\top p_2\\
    &-(\frac{\partial^2 g_2}{\partial x\partial v}\frac{\partial v^*}{\partial p_2})^\top y+(\frac{\partial v^*}{\partial p_2})^\top\frac{\partial g_1}{\partial v}\}dt-\frac{\partial \sigma}{\partial x}ydW(t),\\
    dp_1(t)=&-\frac{\partial H_1}{\partial x}dt+q_1dW(t)\\
    =&-\{\frac{\partial f}{\partial x}+\frac{\partial f}{\partial v}\frac{\partial v^*}{\partial x}+(\frac{\partial\sigma}{\partial x})^\top q_1+\frac{\partial g_1}{\partial x}+(\frac{\partial v^*}{\partial x})^\top\frac{\partial g_1}{\partial v}\\
    &-\sum_iy_i[\frac{\partial}{\partial x}(\frac{\partial f}{\partial x_i})^\top+(\frac{\partial v^*}{\partial x})^\top\frac{\partial}{\partial v}(\frac{\partial f}{\partial x_i})^\top]p_2\\
    &-\sum_iy_i\frac{\partial}{\partial x}(\frac{\partial \sigma}{\partial x_i})^\top q_2-(\frac{\partial^2g_2}{\partial x^2}+\frac{\partial^2g_2}{\partial x\partial v}\frac{\partial v^*}{\partial x})^\top y\}dt+q_1dW(t),\\
    y(0)=&0,\ \ p_1(T)=-\frac{\partial^2G_2}{\partial x^2}(x(T))y(T)+\frac{\partial G_1}{\partial x}(x(T)).
    \end{split}\right.
\end{equation}
\section{Stochastic closed-loop Stackelberg games}\label{sec3}
In this section, we consider a stochastic closed-loop Stackelberg
game which is a stochastic version of the paper
\cite{PapavassiCruz79}. The difference between open-loop Stackelberg
games and closed-loop Stackelberg games is that in the former case
the leader's information set is the $\sigma$-field $\mathcal {F}_t$
generated by the Brownian motion $W$, whereas in the latter case the
leader's information set involves both the $\sigma$-field $\mathcal
{F}_t$ and the history of the state $x$. As stated in the
introduction, the difficulty of studying closed-loop Stackelberg
games arises from the fact that the reaction of the follower can not
be determined explicitly if the leader's strategy depends on the
whole history of the state (CLPS information structure). However, if
the leader's strategy is restricted to be memoryless, i.e., only the
current state is involved in the strategy, Papavassilopoulos and
Cruz \cite{PapavassiCruz79} provide an efficient way to solve such a
problem.
As demonstrated in \cite{PapavassiCruz79}, the derivative $\frac{\partial u}{\partial x}$ of the leader's strategy $u$ will appear in the follower's adjoint equation and further in the leader's augmented state equation, which makes the leader's control problem a nonclassical one.
\subsection{The deterministic case revisited}
Since we apply the approach in Papavassilopoulos and Cruz
\cite{PapavassiCruz79} to solve the stochastic version of
closed-loop Stackelberg games, we fist elaborate their techniques in
this subsection. The state and the cost functionals for the leader
and the follower are as follows
\begin{equation}\label{gg3}
\left\{
\begin{split}
  \dot{x}(t)&=f(t,x(t),u(t),v(t)),\\
  x(0)&=x_0,
\end{split}\right.
\end{equation}
\begin{equation}\label{gg4}
\begin{split}
  J_1(u,v)&=\int_0^T g_1(t,x(t),u(t),v(t))dt+G_1(x_T),\\
  J_2(u,v)&=\int_0^T g_2(t,x(t),u(t),v(t))dt+G_2(x_T).
\end{split}
\end{equation}
Given the leader's strategy $u(t,x)_{t\in[0,T]}$ (we omit to write
the dependence on the initial state $x_0$) which is continuously
differentiable in $x$, if the follower's optimal response is $v^*$,
then according to the deterministic maximum principle, there exists
a function $p$ such that
\begin{equation}\label{gg5}
\left\{
  \begin{split}
  &\dot{x}=f(t,x,u,v^*),\\
  &-\dot{p}=(\frac{\partial f}{\partial x}+\frac{\partial f}{\partial u}\frac{\partial u}{\partial x})^\top p+\frac{\partial g_2}{\partial x}+(\frac{\partial u}{\partial x})^\top\frac{\partial g_2}{\partial u},\\
  &\frac{\partial g_2}{\partial v}+\frac{\partial f}{\partial v}p=0,\\
  &x(0)=x_0,\ p(T)=\frac{\partial G_2(x(T))}{\partial x}.
  \end{split}\right.
\end{equation}
Suppose we can get the unique solution
\begin{equation}\label{gg6}
  v=\varphi(t,x,p,u)
\end{equation}
from solving
$$\frac{\partial g_2}{\partial v}+\frac{\partial f}{\partial v}p=0.$$
Then, after substituting the expression \eqref{gg6} into \eqref{gg5} and $J_1$, the leader will be faced with the following problem
\begin{equation}\label{gg7}
    \min_{u} J_1(u)=\int_0^T g_1(t,x,u,\varphi(t,x,p,u))dt+G_1(x_T)
\end{equation}
subject to
\begin{equation}\label{gg8}
\left\{
  \begin{split}
      \dot{x}&=f(t,x,u,\varphi(t,x,p,u)),\\
  -\dot{p}&=[\frac{\partial f}{\partial x}+\frac{\partial f}{\partial u}\frac{\partial u}{\partial x}]^\top p+\frac{\partial g_2}{\partial x}+(\frac{\partial u}{\partial x})^\top\frac{\partial g_2}{\partial u},\\
   x(0)&=x_0,\ p(T)=\frac{\partial G_2(x(T))}{\partial x}.
  \end{split}\right.
\end{equation}
Since the derivative $\frac{\partial u}{\partial x}$ of the control
variable $u$ is involved in the adjoint equation \eqref{gg8}, the
above problem is a nonclassical one. The authors provide two
approaches to overcome this difficulty. One is the direct
application of variational techniques. The other one is more
interesting, which reveals the relative independence of $u$ and
$\frac{\partial u}{\partial x}$ and the time inconsistency property.
To be more precise, with $\frac{\partial u}{\partial x}$ replaced by
another new control variable $\tilde{u}$, they construct a new
classical problem
\begin{equation}\label{gg9}
    \min_{u,\tilde{u}} \tilde{J}_1(u)=\int_0^T g_1(t,x,u,\varphi(t,x,p,u))dt+G_1(x_T)
\end{equation}
subject to
\begin{equation}\label{gg10}
\left\{
  \begin{split}
      \dot{x}&=f(t,x,u,\varphi(t,x,p,u)),\\
  -\dot{p}&=(\frac{\partial f}{\partial x}+\frac{\partial f}{\partial u}\tilde{u})^\top p+\frac{\partial g_2}{\partial x}+(\tilde{u})^\top\frac{\partial g_2}{\partial u},\\
   x(0)&=x_0,\ p(T)=\frac{\partial G_2(x(T))}{\partial x},
  \end{split}\right.
\end{equation}
and prove the equivalence of the above nonclassical problem
\eqref{gg7}-\eqref{gg8} and the constructed classical problem \eqref{gg9}-\eqref{gg10} in the
sense that they have the same optimal trajectory and costs.
Indeed, if we denote by $J_1^*$ and $J_2^*$ the optimal values of problems \eqref{gg7}-\eqref{gg8} and \eqref{gg9}-\eqref{gg10}, respectively, then $J_1^*\geq J_2^*$.
On the other hand, suppose that $(u^*,\tilde{u}^*)$ is an optimal control for problem \eqref{gg9}-\eqref{gg10}
and $x^*$ is the corresponding trajectory, then control
\begin{equation}\label{gg11}
  \hat{u}(t,x):=\tilde{u}^*(t)x+u^*(t)-\tilde{u}^*(t)x^*(t)
\end{equation}
yields the same trajectory $x^*$ and thus the same cost in problem \eqref{gg7}-\eqref{gg8}. Consequently, $J_1^*=J_2^*$ and $\hat{u}$ is an optimal control for the nonclassical problem \eqref{gg7}-\eqref{gg8}. Therefore, one can substitute $\frac{\partial u}{\partial x}$ for
$\tilde{u}$ in the maximum principle for the problem \eqref{gg9}-\eqref{gg10} and
finally get the maximum principle for the nonclassical problem
\eqref{gg7}-\eqref{gg8} faced by the leader.
\begin{rem}
  Given the leader's strategy $u(t,x)_{t\in[0,T]}$, the follower can also solve the following Hamilton-Jacobi-Bellman equation
  \begin{equation}\label{gg12}
    \left\{
    \begin{split}
      &\frac{\partial V_2}{\partial t}+\inf_{v\in\mathbb{R}^n}\{\langle\frac{\partial V_2}{\partial x}, f(t,x,u(t,x),v)\rangle+g_2(t,x,u(t,x),v)\}=0,\\
      &V_2(T,x)=G_2(x),
    \end{split}\right.
  \end{equation}
  and obtain the optimal feedback strategy
  $$v^*(t,x)=arg\inf_{v\in\mathbb{R}^n}\{\langle\frac{\partial V_2}{\partial x}, f(t,x,u(t,x),v)\rangle+g_2(t,x,u(t,x),v)\}.$$
  However, since $V_2$ depends on the whole function $u(\cdot)$, it is impossible for the leader to employ dynamic programming to depict his optimal strategy. The maximum principle approach turns out to be more appropriate for closed-loop Stackelberg games.
\end{rem}
\subsection{The stochastic case}\label{sgs}
In this subsection we tackle closed-loop Stackelberg games in the stochastic context, with the same idea as \cite{PapavassiCruz79}. After introducing a stochastic disturbance term in the state equation \eqref{gg3}, the adjoint equation for the follower, which also acts as the state equation in the leader's problem, will be a BSDE rather than an ODE with a terminal condition. Therefore, the leader will end up with a control problem in which the state equation consists of a SDE and a BSDE, with the feature that both the control $u$ and its derivative $\frac{\partial u}{\partial x}$ are introduced in the controlled system. With the results on the maximum principle for control problems of FBSDEs, we present the necessary conditions for the leader's optimal strategy to satisfy in a closed-loop Stackelberg game.

We first introduce the admissible strategy spaces for the leader and the follower
\begin{equation*}
\begin{split}
  \mathcal {U}&:=\{u:u:\Omega\times[0,T]\times\mathbb{R}^n\rightarrow U\ \textrm{is}\ \mathcal {F}_t\textrm{-adapted for any}\ x\in\mathbb{R}^n, u(t,x)\ \textrm{is continuously}\\
   &\ \ \textrm{differentible in}\ x\ \textrm{for any}\ (\omega,t)\in\Omega\times[0,T],\ \textrm{and the derivative}\ \frac{\partial u}{\partial x}\ \textrm{is bounded}\},\\
  \mathcal {V}&:=\{v:v:\Omega\times[0,T]\times\mathbb{R}^n\rightarrow V\ \textrm{is}\ \mathcal {F}_t\textrm{-adapted for any}\ x\in\mathbb{R}^n\}.
\end{split}
\end{equation*}
Then, given the leader's strategy $u(t,x)$, the follower's optimal response strategy $v^*(t,x)$ is a solution to the following classical optimal control problem,
\begin{equation}\label{g1}
  \min_{v\in\mathcal {V}} J_2=E\int_0^Tg_2(t,x(t),u(t,x(t)),v(t))dt+EG_2(X(T)),
\end{equation}
subject to
\begin{equation}
\left\{
  \begin{split}\label{g2}
    dx(t)&~=f(t,x(t),u(t,x(t)),v(t))dt+\sigma(t,x(t))dW(t),\\
    x(0)&~=x_0.
  \end{split}\right.
\end{equation}
According to the maximum principle, there exists a pair of adapted processes $(p_2, q_2)\in\mathcal{S}^2\times\mathcal{M}^2$ such that
\begin{equation}\label{g3}
    v^*(t,x(t))=arg \min_{v\in V}\{\langle p_2(t),f(t,x(t),u(t,x(t)),v)\rangle+\langle q_2,\sigma(t,x)\rangle+g_2(t,x(t),u(t,x(t)),v)\},
\end{equation}
and
\begin{equation}\label{g4}
\left\{
  \begin{split}
    dp_2(t)=&-[(\frac{\partial f}{\partial x}+\frac{\partial f}{\partial u}\frac{\partial u}{\partial x})^\top p_2+(\frac{\partial \sigma}{\partial x})^\top q_2\\
    &+\frac{\partial g_2}{\partial x}+(\frac{\partial u}{\partial x})^\top\frac{\partial g_2}{\partial u}]dt+q_2(t)dW(t),\\
    p_2(T)=&\frac{\partial G_2}{\partial x}(x(T)),
  \end{split}\right.
\end{equation}
where $x(\cdot)$ is the solution of \eqref{g2} with policies
$u(t,x)$ and $v^*(t,x)$. Suppose for any leader's strategy $u(t,x)$,
there exists a unique strategy $v^*(t,x)$ for the follower that
minimizes his cost functional $J_2$. We also suppose that \eqref{g3}
yields $v^*=\varphi(t,x,u,p_2)$. Then, taking into account the
follower's optimal response, the leader will be confronted with the
optimal control problem
\begin{equation}\label{g5}
  \min_{u\in\mathcal {U}} J_1=E\int_0^Tg_1(t,x(t),u(t,x(t)),\varphi(t,x(t),u(t,x(t)),p_2(t)))dt+EG_1(x(T))
\end{equation}
subject to
\begin{equation}\label{g6}
\left\{
  \begin{split}
    dx(t)=&f(t,x(t),u(t,x(t)),\varphi(t,x(t),u(t,x(t)),p_2(t)))dt+\sigma(t,x(t))dW(t),\\
    dp_2(t)=&-[(\frac{\partial f}{\partial x}+\frac{\partial f}{\partial u}\frac{\partial u}{\partial x})^\top p_2+(\frac{\partial \sigma}{\partial x})^\top q_2\\
    &+\frac{\partial g_2}{\partial x}+(\frac{\partial u}{\partial x})^\top\frac{\partial g_2}{\partial u}]dt+q_2(t)dW(t),\\
    x(0)=&x_0,\ \ p_2(T)=\frac{\partial G_2}{\partial x}(x(T)).
  \end{split}\right.
\end{equation}
It can be seen that, after incorporating the follower's adjoint
variable as an augmented state, the leader encounters a controlled
FBSDE, which is the counterpart of \eqref{gg8} in the deterministic
context. For the solvability of FBSDEs, one can refer to
\cite{Maetal94}, \cite{PengWu99}, \cite{PardouxTang99},
\cite{Yong10a}, and the references therein. Here we assume that the
leader's problem is well-posed, i.e., for each $u(\cdot)\in\mathcal
{U}$, there exists a unique triple
$(x,p_2,q_2)\in\mathcal{S}^2\times\mathcal{S}^2\times\mathcal{M}^2$
solving FBSDE \eqref{g6}. Since the derivative $\frac{\partial
u}{\partial x}$ of the control variable $u$ is involved in the BSDE
in \eqref{g6}, we apply the techniques in the deterministic case to
relate the above nonclassical control problem to a classical one.

Consider the optimization problem of a controlled FBSDE
\begin{equation}\label{g7}
  \min_{u_1,u_2} J(u_1(\cdot),u_2(\cdot))=E\int_0^Tg_1(t,x(t),u_1(t),\varphi(t,x(t),u_1(t),p_2(t)))dt+EG_1(x(T)),
\end{equation}
subject to
\begin{equation}\label{g8}
\left\{
  \begin{split}
    dx(t)~=&f(t,x(t),u_1(t),\varphi(t,x(t),u_1(t),p_2(t)))dt+\sigma(t,x(t))dW(t),\\
    dp_2(t)~=&-[(\frac{\partial f}{\partial x}+\frac{\partial f}{\partial u}u_2)^\top p_2+(\frac{\partial \sigma}{\partial x})^\top q_2\\
    &+\frac{\partial g_2}{\partial x}+(u_2)^\top\frac{\partial g_2}{\partial u}]dt+q_2(t)dW(t),\\
    x(0)~=&x_0,\ p_2(T)=\frac{\partial G_2}{\partial x}(x(T)),
  \end{split}\right.
\end{equation}
where $u_1$ and $u_2$ are adapted control variables with values in $U$ and some bounded subset in $\mathbb{R}^{m_1\times n}$, respectively. 
Again we assume the above problem is well-posed. Obviously, if we denote by $J_1^*$ and $J^*$ the optimal values of problems \eqref{g5}-\eqref{g6} and \eqref{g7}-\eqref{g8}, respectively, then $J_1^*\geq J^*$. On the other hand, if $(u_1^*,u_2^*)$ is a solution to problem \eqref{g7}-\eqref{g8} and $x^*$ is the corresponding optimal state trajectory, then we can construct an optimal control $u^*$ for problem \eqref{g5}-\eqref{g6} as follows
\begin{equation}\label{g9}
  u^*(t,x):=u_2^*(t)x+u_1^*(t)-u_2^*(t)x^*(t).
\end{equation}
Therefore, $J_1^*=J^*$, which implies that if $u^*(t,x)$ is a solution to problem \eqref{g5}-\eqref{g6} and $x^*$ is the corresponding optimal state trajectory, then $(u^*(t,x^*(t)),\frac{\partial u^*}{\partial x}(t,x^*(t)))$ is an optimal control for problem \eqref{g7}-\eqref{g8} and leads to the same optimal state trajectory $x^*$. Thus we can obtain the maximum principle for problem \eqref{g5}-\eqref{g6} faced by the leader by means of the necessary conditions satisfied by the optimal control for problem \eqref{g7}-\eqref{g8} (see, e.g., \cite{ShiWu06} or \cite{Yong10b}). To this end, we
define
\begin{equation}\label{g10}
   \begin{split}
     &H_1(t,u_1,u_2,x,y,p_1,p_2,q_1,q_2)\\
     =~&\langle p_1, f(t,x,u_1,\varphi(t,x,u_1,p_2))\rangle+\langle q_1,\sigma(t,x)\rangle-\langle y,(\frac{\partial f}{\partial x}+\frac{\partial f}{\partial u}u_2)^\top p_2\\
     &+(\frac{\partial \sigma}{\partial x})^\top q_2+\frac{\partial g_2}{\partial x}+(u_2)^\top\frac{\partial g_2}{\partial u}\rangle+g_1(t,x,u_1,\varphi(t,x,u_1,p_2)).
    \end{split}
\end{equation}
\begin{thm}
  Suppose $u^*(t,x)$ is a solution to the leader's problem \eqref{g5}-\eqref{g6}.
   Then there exists a triple $(y,p_1,q_1)$ such that
  \begin{equation}\label{g11}
  \begin{split}
    &(u^*(t,x(t)),\frac{\partial u^*}{\partial x}(t,x(t)))\\
    =&arg_{(u^1,u^2)}\min H_1(t,u^1,u^2,x(t),y(t),p_1(t),p_2(t),q_1(t),q_2(t))
  \end{split}
  \end{equation}
  and
  \begin{equation}\label{g12}
  \left\{
    \begin{split}
      dy(t)=&-\frac{\partial H_1}{\partial p_2}dt-\frac{\partial H_1}{\partial q_2}dW(t),\\
      dp_1(t)=&-\frac{\partial H_1}{\partial x}dt+q_1(t)dW(t),\\
      y(0)=&~0,\ \ p_1(T)=-\frac{\partial^2 G_2}{\partial x^2}(x(T))y(T)+\frac{\partial G_1}{\partial x}(x(T)),
    \end{split}\right.
  \end{equation}
  where $(x,p_2,q_2)$ is the solution of state equation \eqref{g6} with control $u^*(t,x)$, and $\frac{\partial H_1}{\partial p_2}$, $\frac{\partial H_1}{\partial q_2}$ and $\frac{\partial H_1}{\partial x}$ in \eqref{g12} are evaluated at
  $$(t,u^*(t,x(t)),\frac{\partial u^*}{\partial x}(t,x(t)),x(t),y(t),p_1(t),p_2(t),q_1(t),q_2(t)).$$
\end{thm}
\begin{rem}
  If $u$ is independent of $x$, we conclude in comparison with the arguments in section \ref{open} that the closed-loop Stackelberg solution is
reduced to the open-loop Stackelberg solution and the maximum
principles for both cases are identical.
\end{rem}
\section{The linear quadratic Stackelberg games}
In this section we consider linear quadratic open-loop and
closed-loop Stackelberg games. Yong derives the Riccati equation for
the open-loop Stackelberg game in \cite{Yong02} where the weighting
matrices of the state and controls in the cost functionals are
assumed not necessarily positive definite, and controls are allowed
to appear in the diffusion term. For the follower's problem, the
author uses the solutions of the follower's Riccati equation and a
BSDE to give the state feedback representation of the follower's
optimal strategy (one can also refer to \cite[Page 313]{YongZhou99}
for a similar derivation of the state feedback representation for a
linear quadratic stochastic control problem with deterministic
coefficients). To be precise, the author assumes that the follower's
adjoint variable $p_2$ in \eqref{o9} has the affine form
$$p_2=Px+\phi.$$
Applying It\^{o}'s formula to $p_2$ and taking into account
\eqref{o7} and \eqref{o9}, one can get the follower's Riccati
equation with respect to $P$ and a BSDE for $\phi$. Then the author
views the above BSDE for $\phi$, which contains the solution of the
follower's Riccati equation and the leader's adopted strategy, and
the original state equation as the leader's controlled system and
further derives the leader's Riccati equation. Under some
assumptions the author also discusses the solvability of the Riccati
equations for the case of deterministic coefficients. Here we
consider the follower's Hamiltonian system \eqref{o10} as the leader's
controlled state equation and hence the state feedback
representation of the Stackelberg solution can be obtained at the
same time for the leader and the follower. As a result, the
corresponding Riccati equation here is of different form from the
one in \cite{Yong02}. Since we deal with the case without decision
variables in the diffusion term, we also show, under some
appropriate assumptions, the existence and uniqueness of the
solution to the derived Riccati equation with stochastic
coefficients by means of a linear transformation to the standard
stochastic Riccati equation. For the linear quadratic closed-loop
Stackelberg game, we will see that the Hamiltonian system for the
leader is no longer linear, which prevents us from getting an
exogenous Riccati equation if we proceed the same way as in the
open-loop case. Instead, we assume that the forward variable $y$ is
linear with respect to the original state $x$ and derive an
exogenous FBSDE which plays the same role as the Riccati equation in
open-loop case. Throughout this section we assume the coefficients
$A,B_i,C,Q_i,R_i,G_i$ are adapted bounded matrices, $Q_i,R_i,G_i$
are symmetric and nonnegative, and $R_i$ are uniformly positive,
$i=1,2$.
\subsection{The open-loop case}
The state equation and cost functionals are given as follows.
\begin{equation}\label{o7}
\left\{
\begin{split}
 dx(t)&=(Ax+B_1u+B_2v)dt+CxdW(t),\\
 x(0)&=x_0,
\end{split}\right.
\end{equation}
\begin{equation}\label{o8}
\begin{split}
J_1(u,v)&=\frac{1}{2}E[\int_0^T(\langle Q_1x(t),x(t)\rangle+\langle R_1u(t),u(t)\rangle)dt+\langle G_1x(T),x(T)\rangle],\\
J_2(u,v)&=\frac{1}{2}E[\int_0^T(\langle Q_2x(t),x(t)\rangle+\langle R_2v(t),v(t)\rangle)dt+\langle G_2x(T),x(T)\rangle].
\end{split}
\end{equation}
Given leader's strategy $u\in\mathcal {U}$, it is well known that the follower's problem
$$\min_{v\in\mathcal {V}}\ J_2(u,v)=\frac{1}{2}E[\int_0^T(\langle Q_2x(t),x(t)\rangle+\langle R_2v(t),v(t)\rangle)dt+\langle G_2x(T),x(T)\rangle]$$
subject to
\begin{equation*}
\left\{
\begin{split}
  dx(t)&=(Ax+B_1u+B_2v)dt+CxdW(t),\\
   x(0)&=x_0,
\end{split}\right.
\end{equation*}
is a standard linear quadratic optimal control problem and the unique solution is
$$v^*(t)=-R_2^{-1}B_2^\top p_2,$$
where $p_2$ is the first part of the solution $(p_2,q_2)\in\mathcal{S}^2\times\mathcal{M}^2$ to the adjoint equation
\begin{equation}\label{o9}
\left\{
\begin{split}
  -dp_2(t)&=(A^\top p_2+C^\top q_2+Q_2x)dt-q_2dW(t),\\
  p_2(T)&=G_2x(T).
\end{split}\right.
\end{equation}
Then, the leader's problem is
$$\min_{u\in\mathcal {U}}\ J_1(u)=\frac{1}{2}E[\int_0^T(\langle Q_1x(t),x(t)\rangle+\langle R_1u(t),u(t)\rangle)dt+\langle G_1x(T),x(T)\rangle]$$
subject to (the Hamiltonian system of the follower)
\begin{equation}\label{o10}
  \left\{
  \begin{split}
    dx(t)&=(Ax+B_1u-B_2R_2^{-1}B_2^\top p_2)dt+CxdW(t),\\
  -dp_2(t)&=(A^\top p_2+C^\top q_2+Q_2x)dt-q_2dW(t),\\
  x(0)&=x_0,\ p_2(T)=G_2x(T).
 \end{split}\right.
\end{equation}
The leader's problem is well-posed since for every $u\in\mathcal
{U}$, the coefficients of the system \eqref{o10} satisfy the
monotonicity condition proposed by Peng and Wu \cite{PengWu99},
which yields the existence and uniqueness of the solution
$(x,p_2,q_2)$ to the system \eqref{o10}. Moreover, by similar
arguments of Tang \cite{Tang03}, we can get the following estimate
\begin{equation}\label{o11}
  E\sup_{0\leq t\leq T}|p_2(t)|^2+E\sup_{0\leq t\leq T}|x(t)|^2+E\int_0^T|q_2(t)|^2dt\leq L(|x_0|^2+E\int_0^T|u(t)|^2dt),
\end{equation}
where $L$ is a positive constant. With this estimate, we can adopt
relevant arguments for standard linear quadratic optimal control
problems in \cite{Meng11} and get the fact that the leader's
objective functional $J_1(u)$ is convex in $u$,
$$\lim_{\|u\|\rightarrow\infty}J_1(u)=\infty,$$
and $J_1(u)$ is Fr\'{e}chet differentiable over $\mathcal {U}$ with the representation
\begin{equation}
  \begin{split}
\langle J_1'(u),w\rangle=&E\int_0^T(\langle Q_1(t)x(t;x_0,u),x(t;0,w)\rangle+\langle R_1(t)u(t),w(t)\rangle)dt\\
    &+\langle G_1x(T;x_0,u),x(T;0,w)\rangle.
  \end{split}
\end{equation}
Here we use $x(\cdot;x_0,u)$ to represent the solution of
\eqref{o10} with initial state $x(0)=x_0$ and control $u$. As a
conclusion of Proposition 2.1.2 in \cite{EkelandTeman76}, we know
that the leader has a unique optimal strategy $u^*\in\mathcal {U}$
which satisfies $J_1'(u^*)=0$. Now we use dual representation to
characterize the optimal strategy $u^*$.
\begin{thm}
  For each $u\in\mathcal {U}$, there exists a unique solution $(x,y,p_1,q_1,p_2,q_2)$ to the FBSDE
  \begin{equation}\label{o12}
  \left\{
    \begin{split}
      dx(t)&=(Ax+B_1u-B_2R_2^{-1}B_2^\top p_2)dt+CxdW(t),\\
  -dp_2(t)&=(A^\top p_2+C^\top q_2+Q_2x)dt-q_2dW(t),\\
  dy(t)&=(Ay+B_2R_2^{-1}B_2^\top p_1)dt+CydW(t),\\
  -dp_1(t)&=(A^\top p_1+C^\top q_1-Q_2y+Q_1x)dt-q_1dW(t),\\
  x(0)&=x_0,\ y(0)=0,\ p_1(T)=-G_2y(T)+G_1x(T),\ p_2(T)=G_2x(T).
    \end{split}\right.
  \end{equation}
  The necessary and sufficient condition for $u$ to be the leader's optimal strategy is
  $$u(t)=-R_1^{-1}B_1p_1(t).$$
  \end{thm}
  \begin{proof}
    It can be seen that the FBSDEs consisting of $(x,p_2,q_2)$ and $(y,p_1,q_1)$ are two decoupled systems.
Therefore, for given $u\in\mathcal {U}$, we can first get the unique
solution $(x,p_2,q_2)$ to the equation
    \begin{equation}\label{o13}
      \left\{
      \begin{split}
    dx(t)&=(Ax+B_1u-B_2R_2^{-1}B_2^\top p_2)dt+CxdW(t),\\
    -dp_2(t)&=(A^\top p_2+C^\top q_2+Q_2x)dt-q_2dW(t),\\
    x(0)&=x_0,\  p_2(T)=G_2x(T).
      \end{split}\right.
    \end{equation}
    Let $\tilde{y}:=-y$. Then FBSDE consisting of $(y,p_1,q_1)$ in \eqref{o12} can be converted into the following one
    \begin{equation}\label{o14}
      \left\{
      \begin{split}
        d\tilde{y}(t)&=(A\tilde{y}-B_2R_2^{-1}B_2^\top p_1)dt+C\tilde{y}dW(t),\\
        -dp_1(t)&=(A^\top p_1+C^\top q_1+Q_2\tilde{y}+Q_1x)dt-q_1dW(t),\\
        \tilde{y}(0)&=0,\ p_1(T)=G_2\tilde{y}(T)+G_1x(T).
      \end{split}\right.
    \end{equation}
    The coefficients in the above system also satisfy the monotonicity condition in \cite{PengWu99}. So there
exists a unique solution to \eqref{o14}, which also implies the
existence and uniqueness of the solution $(x,y,p_1,q_1,p_2,q_2)$ to
FBSDE \eqref{o12}. The necessary part comes directly from the
maximum principle \eqref{oo6} and \eqref{o6}. Now we prove the
sufficient part. Denote by
$$(x(\cdot;x_0,u),y(\cdot;x_0,u),p_1(\cdot;x_0,u),q_1(\cdot;x_0,u),p_2(\cdot;x_0,u),q_2(\cdot;x_0,u))$$
    and
    $$(x(\cdot;0,w),y(\cdot;0,w),p_1(\cdot;0,w),q_1(\cdot;0,w),p_2(\cdot;0,w),q_2(\cdot;0,w))$$
    the solutions to the system of FBSDEs \eqref{o12} with initial states and controls as $(x_0,u)$ and $(0,w)$, respectively. Using It\^{o}'s formula to compute
    $$\langle p_1(t;x_0,u),x(t;0,w)\rangle+\langle p_2(t;0,w),y(t;x_0,u)\rangle$$
    and taking the expectation, we can get
    \begin{equation}\label{o15}
      \begin{split}
      \langle J_1'(u),w\rangle
      =&E\langle G_1x(T;x_0,u),x(T;0,w)\rangle\\
      &+E\int_0^T\langle Q_1(t)x(t;x_0,u),x(t;0,w)\rangle +\langle R_1(t)u(t),w(t)\rangle dt\\
      =&E\int_0^T\langle R_1(t)u(t)+B_1^\top(t)p_1(t;x_0,u),w(t)\rangle dt.
      \end{split}
    \end{equation}
    Obviously $u=-R_1^{-1}B_1^\top p_1$ makes $J_1'(u)$ equal to zero, so it is an optimal strategy for the leader.
  \end{proof}
From the uniqueness of the optimal strategy, we also know that FBSDE
\begin{equation}\label{o16}
  \left\{
    \begin{split}
  dx(t)&=(Ax-B_1R_1^{-1}B_1^\top p_1-B_2R_2^{-1}B_2^\top p_2)dt+CxdW(t),\\
  -dp_2(t)&=(A^\top p_2+C^\top q_2+Q_2x)dt-q_2dW(t),\\
  dy(t)&=(Ay+B_2R_2^{-1}B_2^\top p_1)dt+CydW(t),\\
  -dp_1(t)&=(A^\top p_1+C^\top q_1-Q_2y+Q_1x)dt-q_1dW(t),\\
  x(0)&=x_0,\ y(0)=0,\ p_1(T)=-G_2y(T)+G_1x(T),\ p_2(T)=G_2x(T),
    \end{split}\right.
\end{equation}
has a unique solution $(x,y,p_1,q_1,p_2,q_2)$. And the Stackelberg solution $(u^*,v^*)$ can be written as
\begin{equation}\label{oo16}
  u^*=-R_1^{-1}B_1^\top p_1,\ \ v^*=-R_2^{-1}B_2^\top p_2.
\end{equation}
In what follows we see $(x,y)$ as the state and derive the feedback representation of the Stackelberg solution $(u^*,v^*)$ in terms of $(x,y)$.
We denote
\begin{equation*}
  \hat{x}=\left(
  \begin{array}{c}
  x\\
  y\\
  \end{array}\right),
  \hat{p}=\left(
  \begin{array}{c}
  p_1\\
  p_2\\
  \end{array}\right),
  \hat{q}=\left(
  \begin{array}{c}
  q_1\\
  q_2\\
  \end{array}\right),
\end{equation*}
and
\begin{equation*}
  \begin{split}
  \hat{A}=&\left(
  \begin{array}{cc}
  A&0\\
  0&A\\
  \end{array}\right),
  \hat{B}=\left(
  \begin{array}{cc}
  B_1R_1^{-1}B_1^\top&B_2R_2^{-1}B_2^\top\\
  -B_2R_2^{-1}B_2^\top&0\\
  \end{array}\right),
  \hat{C}=\left(
  \begin{array}{cc}
  C&0\\
  0&C\\
  \end{array}\right),\\
    \hat{Q}=&\left(
  \begin{array}{cc}
  Q_1&-Q_2\\
  Q_2&0\\
  \end{array}\right),
      \hat{G}=\left(
  \begin{array}{cc}
  G_1&-G_2\\
  G_2&0\\
  \end{array}\right).
  \end{split}
\end{equation*}
Then FBSDE \eqref{o16} can be rewritten as
\begin{equation}\label{o17}
  \left\{
  \begin{split}
    d\hat{x}(t)&=(\hat{A}\hat{x}(t)-\hat{B}\hat{p}(t))dt+\hat{C}\hat{x}dW(t),\\
    d\hat{p}(t)&=-(\hat{A}^\top\hat{p}+\hat{C}^\top\hat{q}+\hat{Q}\hat{x})dt+\hat{q}dW(t),\\
    \hat{x}(0)&=0, \ \hat{p}(T)=\hat{G}\hat{x}(T).
  \end{split}\right.
\end{equation}
Suppose there is a matrix-valued process $K$ such that
\begin{equation}\label{oo18}
  \hat{p}=K\hat{x},
\end{equation}
and $K$ has a stochastic differential form
\begin{equation}\label{o18}
  dK(t)=M(t)dt+L(t)dW(t).
\end{equation}
Applying It\^{o}'s formula to $K\hat{x}$, we get
\begin{equation}\label{o19}
  \begin{split}
    &M\hat{x}dt+L\hat{x}dW(t)+K(\hat{A}\hat{x}-\hat{B}K\hat{x}(t))dt+K\hat{C}\hat{x}dW(t)+L\hat{C}\hat{x}dt\\
    =&d\hat{p}(t)\\
    =&-(\hat{A}^\top K\hat{x}+\hat{C}^\top\hat{q}+\hat{Q}\hat{x})dt+\hat{q}dW(t).
  \end{split}
\end{equation}
Comparing the diffusion terms in \eqref{o19}, we have
\begin{equation}\label{o20}
  \hat{q}=L\hat{x}+K\hat{C}\hat{x}.
\end{equation}
Substituting the expression into \eqref{o19} and comparing the drift terms, we get
\begin{equation}\label{o21}
\begin{split}
  &M\hat{x}+K(\hat{A}\hat{x}-\hat{B}K\hat{x}(t))+L\hat{C}\hat{x}\\
  =&-\hat{A}^\top K\hat{x}-\hat{C}^\top(L\hat{x}+K\hat{C}\hat{x})-\hat{Q}\hat{x},
\end{split}
\end{equation}
which yields
$$M=-K\hat{A}-\hat{A}^\top K+K\hat{B}K-L\hat{C}-\hat{C}^\top L-\hat{C}^\top K\hat{C}-\hat{Q}.$$
Therefore, we get the Riccati equation
\begin{equation}\label{o22}
\left\{
\begin{split}
  dK(t)&=-(K\hat{A}+\hat{A}^\top K-K\hat{B}K+L\hat{C}+\hat{C}^\top L+\hat{C}^\top K\hat{C}+\hat{Q})dt+LdW(t),\\
  K(T)&=\hat{G}.
\end{split}\right.
\end{equation}
The difference between the above Riccati equation and the standard
one from stochastic LQ problems without control in diffusion terms
(see, e.g., \cite{Peng92}) is that $\hat{B}$, $\hat{Q}$ and
$\hat{G}$ here are not symmetric matrices. For $n=1$ and under some
appropriate assumptions on the coefficient matrices, we show in the
following proposition that Riccati equation \eqref{o22} can be
connected to a standard one through a linear transformation for
FBSDE \eqref{o17}.
\begin{prop}
Suppose that n=1 and $\alpha$ and $\beta$ are two positive constants such that
$$\frac{Q_2}{Q_1}=\frac{G_2}{G_1}=\alpha,\ \ \frac{B_2R_2^{-1}B_2^\top}{B_1R_1^{-1}B_1^\top}=\beta.$$
Then, the Riccati equation \eqref{o22} has a unique solution.
\end{prop}
\begin{proof}
We make the transformation
\begin{equation}\label{o23}
  \hat{x}=\tilde{x},\ \hat{p}=\Phi\tilde{p},\ \hat{q}=\Phi\tilde{q},
\end{equation}
where
$$\Phi=\left(\begin{array}{cc}
  1&-2\beta\\
  2\alpha&1\\
  \end{array}\right).$$
Then FBSDE \eqref{o17} can be converted into the following one
\begin{equation}\label{o24}
\left\{
  \begin{split}
    d\tilde{x}(t)&=(\tilde{A}\tilde{x}(t)-\tilde{B}\tilde{p}(t))dt+\tilde{C}\tilde{x}dW(t),\\
    d\tilde{p}(t)&=-(\tilde{A}^\top\tilde{p}+\tilde{C}^\top\tilde{q}+\tilde{Q}\tilde{x})dt+\tilde{q}dW(t),\\
    \tilde{x}(0)&=0, \ \tilde{p}(T)=\tilde{G}\tilde{x}(T),
  \end{split}\right.
\end{equation}
where
$$\tilde{A}=\hat{A},\ \tilde{C}=\hat{C},$$
$$\tilde{B}=\left(
  \begin{array}{cc}
  B_1R_1^{-1}B_1^\top+2\alpha B_2R_2^{-1}B_2^\top&-B_2R_2^{-1}B_2^\top\\
  -B_2R_2^{-1}B_2^\top&2\beta B_2R_2^{-1}B_2^\top\\
  \end{array}\right),$$
 $$\tilde{Q}=\frac{1}{4\alpha\beta}\left(
  \begin{array}{cc}
  Q_1+2\beta Q_2&-Q_2\\
  -Q_2&2\alpha Q_2\\
  \end{array}\right),$$
  $$\tilde{G}=\left(
  \begin{array}{cc}
  G_1+2\beta G_2&-G_2\\
  -G_2&2\alpha G_2\\
  \end{array}\right).$$
Now the matrices $\tilde{B}$, $\tilde{Q}$ and $\tilde{G}$ are symmetric and positive definite. Suppose
$$\tilde{p}=\tilde{K}\tilde{x},$$
and $$d\tilde{K}=\tilde{K}_1dt+\tilde{L}dW(t).$$
With the same procedure to derive Riccati equation \eqref{o22}, we can get a standard Riccati equation for $(\tilde{K},\tilde{L})$
\begin{equation}\label{o25}
\left\{
\begin{split}
  d\tilde{K}(t)&=-(\tilde{K}\tilde{A}+\tilde{A}^\top \tilde{K}-\tilde{K}\tilde{B}\tilde{K}+\tilde{L}\tilde{C}+\tilde{C}^\top \tilde{L}+\tilde{C}^\top \tilde{K}\tilde{C}+\tilde{Q})dt+\tilde{L}dW(t),\\
  \tilde{K}(T)&=\tilde{G}.
\end{split}\right.
\end{equation}
According to the results in \cite{Bismut78} or \cite{Peng92}, or more general case in \cite{Tang03}, we know that Riccati equation \eqref{o25} has a unique solution $(\tilde{K},\tilde{L})$ and
$$\tilde{p}=\tilde{K}\tilde{x},\ \tilde{q}=(\tilde{L}+\tilde{K}\tilde{C})\tilde{x}.$$
 Consequently,
\begin{equation}\label{o26}
\begin{split}
\hat{p}&=\Phi\tilde{p}=\Phi\tilde{K}\tilde{x}=\Phi\tilde{K}\hat{x},\\
\hat{q}&=\Phi\tilde{q}=\Phi(\tilde{L}+\tilde{K}\tilde{C})\tilde{x}=\Phi(\tilde{L}+\tilde{K}\tilde{C})\hat{x}.
\end{split}
\end{equation}
Comparing \eqref{o26} with \eqref{oo18} and \eqref{o20}, we finally get
$$K=\Phi\tilde{K},\ L=\Phi\tilde{L}.$$
From \eqref{oo16} we obtain that the Stackelberg solution $(u^*,v^*)$ has a feedback representation in terms of the state $(x,y)$.
\end{proof}
\subsection{The closed-loop case}
 As pointed out in the deterministic case \cite{PapavassiCruz79}, the relative independence
of the leader's strategy $u$ and its derivative $\frac{\partial
u}{\partial x}$ in a closed-loop Stackelberg game makes the leader
so powerful that his Hamiltonian $H$ is likely to achieve $-\infty$
if there is no restriction on the derivative $\frac{\partial
u}{\partial x}$. One way to restrict the leader's strength is to add
a penalty term $\frac{\partial u}{\partial x}$ in his cost
functional in order that $H$ is convex with respect to
$(u,\frac{\partial u}{\partial x})$. The other way is to impose a
prior bounds on $\frac{\partial u}{\partial x}$ to retain $H$
finite. In this section we will adopt the latter way to assume
$\frac{\partial u}{\partial x}$ to be bounded since it will appear
as the coefficient of the unknowns in adjoint equations and the
boundedness of the derivative $\frac{\partial u}{\partial x}$
implies the well-posedness of the leader's problem when affine
strategies are adopted. For simplicity, we consider one-dimensional
linear quadratic game, with the state equation and cost functionals
of the two players as follows
\begin{equation*}
\left\{
\begin{split}
  dx(t)&=[Ax(t)+B_1u(t)+B_2v(t)]dt+Cx(t)dW(t),\\
   x(0)&=x_0,
\end{split}\right.
\end{equation*}
and
\begin{align*}
J_1 &= \frac{1}{2}E [\int_{0}^T(Q_1x^2(t)+R_1u^{2}(t))dt+G_1x^2(T)],\\
J_2 &= \frac{1}{2}E [\int_{0}^T(Q_2x^2(t)+R_2v^{2}(t))dt+G_2x^2(T)].
\end{align*}
The admissible strategy spaces from which the leader and the follower choose their strategies are given by
\begin{equation*}
\begin{split}
  \mathcal {U}&:=\{u|u:\Omega\times[0,T]\times\mathbb{R}\rightarrow U\ \textrm{is}\ \mathcal {F}_t\textrm{-adapted for any}\ x\in\mathbb{R},\ u(t,x)\ \textrm{is continuously}\\
   &\ \ \ \ \ \ \textrm{differentible in}\ x\ \textrm{for any}\ (\omega,t)\in\Omega\times[0,T],\ \textrm{and the derivative}\ |\frac{\partial u}{\partial x}|\leq K\\
   &\ \ \ \ \ \ \textrm{for some postive constant $K$}\},\\
  \mathcal {V}&:=\{v|v:\Omega\times[0,T]\times\mathbb{R}^n\rightarrow V\ \textrm{is }\ \mathcal {F}_t\textrm{-adapted for any}\ x\in\mathbb{R}^n\}.
\end{split}
\end{equation*}
Suppose for leader's each strategy $u\in\mathcal {U}$, the follower has a unique optimal response $v^*\in\mathcal {V}$. From \eqref{g3} we know $$v^*=-R_2^{-1}B_2p_2,$$
with $p_2$ satisfying
\begin{equation*}
\left\{
  \begin{split}
    dp_2(t)&=-[(A+B_1\frac{\partial u}{\partial x})p_2+Cq_2+Q_2x]dt+q_2dW(t),\\
    p_2(T)&=G_2x(T).
  \end{split}\right.
\end{equation*}
Therefore the leader's problem is
\begin{equation}\label{lq2}
  \min_{u\in\mathcal{U}}J_1= \frac{1}{2}E [\int_{0}^T(Q_1x^2(t)+R_1u^{2}(t))dt+G_1x^2(T)]
\end{equation}
subject to
\begin{equation}\label{lq3}
  \left\{
  \begin{split}
    dx(t)&=[Ax(t)+B_1u(t,x(t))-R_2^{-1}B^2_2p_2(t)]dt+Cx(t)dW(t),\\
    dp_2(t)&=-[(A+B_1\frac{\partial u}{\partial x})p_2+Cq_2+Q_2x]dt+q_2dW(t),\\
    x(0)&=x_0,\ p_2(T)=G_2x(T).
  \end{split}\right.
\end{equation}
Suppose that for every $u(t,x)\in\mathcal {U}$, there is a unique
solution $(x,p_2,q_2)$ to FBSDE \eqref{lq3}. According to the
discussions in section \ref{sgs}, we know that the leader will lose
nothing if he chooses his strategy among affine functions
$$u(t,x)=u_2(t)x+u_1(t),$$
with $u_1$ and $u_2$ being adapted processes and $|u_2|\leq K$. Then the leader's equivalent problem can be written as
\begin{equation}\label{lq4}
  \min_{u_1,u_2}\   J_1= \frac{1}{2}E \{\int_{0}^T[Q_1x^2(t)+R_1(u_2(t)x(t)+u_1(t))^2]dt+G_1x^2(T)\}
\end{equation}
subject to
\begin{equation}\label{lq5}
  \left\{
  \begin{split}
  dx(t)&=[(A+B_1u_2)x+B_1u_1-R_2^{-1}B^2_2p_2]dt+Cx(t)dW(t),\\
  dp_2(t)&=-[(A+B_1u_2)p_2+Cq_2+Q_2x]dt+q_2dW(t),\\
  x(0)&=x_0,\ p_2(T)=G_2x(T).
  \end{split}\right.
\end{equation}
For every pair $(u_1,u_2)$, the monotonicity condition guarantees
the existence and uniqueness of the solution to \eqref{lq5}.
Therefore, the leader's problem with strategies restricted being of
affine form is well-posed. In what follows we use the maximum
principle to get the Hamiltonian system and related Riccati equation
for leader's problem \eqref{lq4}-\eqref{lq5}. Denote
\begin{equation}\label{lq6}
  \begin{split}
    &H_1(t,u_1,u_2,x,y,p_1,p_2,q_1,q_2)\\
    =&p_1[(A+B_1u_2)x+B_1u_1-R_2^{-1}B^2_2p_2]+Cxq_1\\
    &-y[(A+B_1u_2)p_2+Cq_2+Q_2x]+\frac{1}{2}[Q_1x^2+R_1(u_2x+u_1)^2].
  \end{split}
\end{equation}
To obtain $(u_1^*,u_2^*)$ that minimizes $H_1(t,u_1,u_2,x,y,p_1,p_2,q_1,q_2)$, we first fix $u_2$ and minimize $H_1$ with respect to $u_1$. By computation,
\begin{equation}\label{lq7}
  u_1^*=-u_2x-R_1^{-1}B_1p_1.
\end{equation}
Substituting \eqref{lq7} into the expression \eqref{lq6} of $H$, we can see the only term containing $u_2$ is
\begin{equation}\label{lq9}
  -B_1yp_2u_2.
\end{equation}
Therefore, the optimal $u_2^*$ is
\begin{equation}\label{lq10}
  u_2^*=\left\{
  \begin{array}{ccc}
  -K, &\ \mbox{if $\Delta>0$},\\
   K, &\ \mbox{if $\Delta<0$},\\
   \mbox{undefined,} &\ \mbox{if $\Delta=0$},
  \end{array}\right.
\end{equation}
where $$\Delta:=-B_1yp_2.$$
To find a candidate of optimal pair $(u_1^*,u_2^*)$, we set
\begin{equation*}
\begin{split}
  u_2^*:=&bang(K,-K;\Delta)\\
  :=&sgn(B_1yp_2)K\\
  =&sgn(y)sgn(B_1p_2)K\\
  =&sgn(p_2)sgn(B_1y)K,
\end{split}
\end{equation*}
where $sgn$ is the sign function defined by
\begin{equation*}
  sgn(x)=\left\{
  \begin{array}{ccc}
  1&\ \mbox{if $x>0$,}\\
  0&\ \mbox{if $x=0$,}\\
  -1&\ \mbox{if $x<0$.}
  \end{array}\right.
\end{equation*}
From \eqref{lq7} we get
\begin{equation}\label{lq11}
  u_1^*=-bang(K,-K;\Delta)x-R_1^{-1}B_1p_1.
\end{equation}
If $(u_1^*,u_2^*)\in\mathcal {U}\times\mathcal {V}$ is a solution to
the leader's problem \eqref{lq4}-\eqref{lq5}, then the maximum
principle yields that there exist adapted processes $y$, $p_1$, and
$q_1$ such that
\begin{equation}\label{lq11}
\left\{
  \begin{split}
  dx(t)&=[(A+B_1u_2^*)x+B_1u_1^*-R_2^{-1}B^2_2p_2]dt+Cx(t)dW(t),\\
  dy(t)&=[(A+B_1u_2^*)y+R_2^{-1}B^2_2p_1]dt+CydW(t),\\
  dp_1(t)&=-[(A+B_1u_2^*)p_1+Cq_1-Q_2y+Q_1x+R_1u_2^*(u_2^*x+u_1^*)]dt+q_1dW(t),\\
  dp_2(t)&=-[(A+B_1u_2^*)p_2+Cq_2+Q_2x]dt+q_2dW(t),\\
  x(0)&=x_0,\ y(0)=0,\ p_1(T)=-G_2y(T)+G_1x(T),\ p_2(T)=G_2x(T),\\
  u_1^*&=-bang(K,-K;\Delta)x-R_1^{-1}B_1p_1,\ u_2^*:=bang(K,-K;\Delta).
  \end{split}\right.
\end{equation}
Like the open-loop case, we proceed to express the optimal strategy
$(u_1^*,u_2^*)$ in a non-anticipating way by means of the state
feedback representation. Substituting the expressions of $u_1^*$ and
$u_2^*$ into the FBSDE in \eqref{lq11}, we get
\begin{equation}\label{glq12}
\left\{
  \begin{split}
  dx(t)&=[Ax-R_1^{-1}B_1^2p_1-R_2^{-1}B^2_2p_2]dt+Cx(t)dW(t),\\
  dy(t)&=[(A+B_1bang(K,-K;\Delta))y+R_2^{-1}B^2_2p_1]dt+CydW(t)\\
       &=[Ay+sgn(p_2)K|B_1y|+R_2^{-1}B^2_2p_1]dt+CydW(t),\\
  dp_1(t)&=-[Ap_1+Cq_1-Q_2y+Q_1x]dt+q_1dW(t),\\
  dp_2(t)&=-[(A+B_1bang(K,-K;\Delta))p_2+Cq_2+Q_2x]dt+q_2dW(t)\\
         &=-[Ap_2+sgn(y)K|B_1p_2|+Cq_2+Q_2x]dt+q_2dW(t),\\
  x(0)&=x_0,\ y(0)=0,\ p_1(T)=-G_2y(T)+G_1x(T),\ p_2(T)=G_2x(T).
  \end{split}\right.
\end{equation}
In contrast to FBSDE \eqref{o16} in the open-loop case, the presence
of the additional nonlinear term $bang(K,-K;\Delta)$ in FBSDE
\eqref{glq12} makes it a nonlinear system. Moreover, the Lipschitz
continuity assumption usually made for the coefficients in the
literature does not hold here. Therefore, the existence and
uniqueness of the solution to \eqref{glq12}, as far as we know, is
still not available. On the other hand, if we still view $(x,y)$ as
the ``state'' and represent $(p_1,p_2)$ in terms of $(x,y)$ as in
the open-loop case, we can not derive an exogenous Riccati equation.
Instead, we only see $x$ as the state and suppose
\begin{equation}\label{lq12}
  y(t)=\xi(t)x(t),\ \ p_1(t)=\eta(t)x(t),\ \ p_2(t)=\zeta(t)x(t),
\end{equation}
and
\begin{equation}\label{lq13}
\begin{split}
d\xi(t)&=\xi_1(t)dt+\xi_2(t)dW(t),\\
d\eta(t)&=\eta_1(t)dt+\eta_2(t)dW(t),\\
d\zeta(t)&=\zeta_1(t)dt+\zeta_2(t)dW(t).
\end{split}
\end{equation}
By It\^{o}'s formula and in view of \eqref{lq12}
\begin{equation}\label{lq14}
\begin{split}
dy(t)=&\xi(t)dx(t)+x(t)d\xi(t)+Cx(t)\xi_2(t)dt\\
     =&\xi(t)[Ax-R_1^{-1}B_1^2p_1-R_2^{-1}B^2_2p_2]dt+C\xi(t)x(t)dW(t)\\
      &+\xi_1(t)x(t)dt+\xi_2(t)x(t)dW(t)+C\xi_2(t)x(t)dt\\
     =&\{[A-R_1^{-1}B_1^2\eta(t)-R_2^{-1}B^2_2\zeta(t)]\xi(t)\\
      &+\xi_1(t)+C\xi_2(t)\}x(t)dt+[C\xi(t)+\xi_2(t)]x(t)dW(t).
\end{split}
\end{equation}

On the other hand,
\begin{equation}\label{lq15}
  \begin{split}
    dy(t)=&[(A+B_1bang(K,-K;\Delta))y+R_2^{-1}B^2_2p_1]dt+CydW(t)\\
  =&[(A+B_1bang(K,-K;\tilde{\Delta}))\xi(t)+R_2^{-1}B^2_2\eta(t)]x(t)dt+C\xi(t)x(t)dW(t),
  \end{split}
\end{equation}
where
$$\tilde{\Delta}:=-B_1\xi(t)\zeta(t).$$
Comparing \eqref{lq14} and \eqref{lq15}, we have
\begin{equation*}
  \begin{split}
\xi_2(t)=&0,\\
\xi_1(t)=&[R_1^{-1}B_1^2\eta(t)+R_2^{-1}B^2_2\zeta(t)+B_1bang(K,-K;\tilde{\Delta})]\xi(t)+R_2^{-1}B^2_2\eta(t).
  \end{split}
\end{equation*}
With the same procedure, we can get
\begin{equation*}
\left\{
  \begin{split}
\eta_1(t)=&[R_1^{-1}B_1^2\eta(t)+R_2^{-1}B^2_2\zeta(t)-2A-C^2]\eta(t)+Q_2\xi(t)-2C\eta_2(t)-Q_1,\\
\zeta_1(t)=&[R_1^{-1}B_1^2\eta(t)+R_2^{-1}B^2_2\zeta(t)-2A-C^2-B_1bang(K,-K;\tilde{\Delta})]\zeta(t)-2C\zeta_2(t)-Q_2.
 \end{split}\right.
\end{equation*}
Therefore, we derive the related Riccati equation for problem \eqref{lq4}-\eqref{lq5}
\begin{equation*}
\left\{
  \begin{split}
d\xi(t)=&\{[R_1^{-1}B_1^2\eta(t)+R_2^{-1}B^2_2\zeta(t)+B_1bang(K,-K;\tilde{\Delta})]\xi(t)+R_2^{-1}B^2_2\eta(t)\}dt\\
       =&\{[R_1^{-1}B_1^2\eta(t)+R_2^{-1}B^2_2\zeta(t)]\xi(t)+sgn(\zeta(t))|B_1\xi(t)|+R_2^{-1}B^2_2\eta(t)\}dt,\\
d\eta(t)=&\{[R_1^{-1}B_1^2\eta(t)+R_2^{-1}B^2_2\zeta(t)-2A-C^2]\eta(t)+Q_2\xi(t)-2C\eta_2(t)\\
         &-Q_1\}dt+\eta_2(t)dW(t),\\
d\zeta(t)=&\{[R_1^{-1}B_1^2\eta(t)+R_2^{-1}B^2_2\zeta(t)-2A-C^2-B_1bang(K,-K;\tilde{\Delta})]\zeta(t)\\
          &-2C\zeta_2(t)-Q_2\}dt+\zeta_2(t)dW(t)\\
         =&\{[R_1^{-1}B_1^2\eta(t)+R_2^{-1}B^2_2\zeta(t)-2A-C^2]\zeta(t)-sgn(\xi(t))|B_1\zeta(t)|\\
          &-2C\zeta_2(t)-Q_2\}dt+\zeta_2(t)dW(t),\\
  \xi(0)=&0,\ \eta(T)=-G_2\xi(T)+G_1,\ \zeta(T)=G_2.
  \end{split}\right.
\end{equation*}
Suppose $(\xi,\eta,\zeta,\eta_2,\zeta_2)$ is a solution to the above FBSDE and $x^*$ solves the linear SDE
\begin{equation*}
\left\{
\begin{split}
 dx(t)&=[A-R_1^{-1}B_1^2\eta-R_2^{-1}B^2_2\zeta]x(t)dt+Cx(t)dW(t),\\
 x(0)&=x_0.
\end{split}\right.
\end{equation*}
Then we can use It\^{o}'s formula to verify that
\begin{equation*}
  \begin{split}
    y(t):=&\xi(t)x^*(t),\ p_1(t):=\eta(t)x^*(t),\ p_2(t):=\zeta(t)x^*(t),\\
    q_1(t):=&[C\eta(t)+\eta_2(t)]x^*(t),\ q_2(t):=[C\zeta(t)+\zeta_2(t)]x^*(t),
  \end{split}
\end{equation*}
together with $x^*$ solve the leader's Hamiltonian system \eqref{glq12}. Therefore,
$$u(t,x)=bang(K,-K;\tilde{\Delta})x-bang(K,-K;\tilde{\Delta})x^*(t)-R_1^{-1}B_1\eta(t)x^*(t)$$
with $\tilde{\Delta}=-B_1\xi(t)\zeta(t)$ is a candidate of the leader's optimal strategy.

\end{document}